\documentclass[a4paper]{amsart}

\usepackage{amscd,amsthm,amsfonts,amssymb,amsmath,esint}
\usepackage[colorlinks=true]{hyperref}
\usepackage[all]{xy}

\usepackage{multirow}
\usepackage{amsthm}
\usepackage{textcomp}
\usepackage{array}
\usepackage{mathrsfs}
\usepackage{enumerate}
\usepackage{comment}
\usepackage{amsfonts}

\theoremstyle{plain} 
\newtheorem{thm}{Theorem}[section]
\newtheorem{prop}[thm]{Proposition}
\newtheorem{lem}[thm]{Lemma}
\newtheorem{cor}[thm]{Corollary}
\newtheorem{conj}[thm]{Conjecture}

\theoremstyle{definition}

\newtheorem{remark}[thm]{Remark}

\newtheorem*{hyp}{Hypothesis}

\newtheoremstyle{TheoremNum}
        {\topsep}{\topsep}             
        {\itshape}                      
        {}                             
        {\bfseries}                    
        {.}                            
        { }                          
        {\thmname{#1}\thmnote{ \bfseries #3}}
\theoremstyle{TheoremNum}

\numberwithin{equation}{section}

\renewcommand{\bold}{\boldsymbol}
\newcommand{\mf}{\mathfrak}

\newcommand{\mc}{\mathcal}
\newcommand{\mb}{\mathbb}

\newcommand{\ovl}{\overline}

\newcommand{\xto}{\xrightarrow}

\newcommand{\F}{\Phi}

\newcommand{\Q}{\mb{Q}}
\newcommand{\Z}{\mb{Z}}

\renewcommand{\O}{\mc{O}}

\DeclareMathOperator{\Gal}{Gal}

\DeclareMathOperator{\Sel}{Sel}
\DeclareMathOperator{\rank}{rank}

\DeclareMathOperator{\ord}{ord}

\DeclareMathOperator{\Hom}{Hom}

\DeclareFontFamily{U}{wncy}{}
\DeclareFontShape{U}{wncy}{m}{n}{<->wncyr10}{}
\DeclareSymbolFont{mcy}{U}{wncy}{m}{n}
\DeclareMathSymbol{\Sha}{\mathord}{mcy}{"58}

\title[Central $L$-values and the growth of {$\Sha[3]$}]{On central $L$-values and the growth of the $3$-part of the Tate--Shafarevich group}
\author{Yukako Kezuka}
\address{Institut de Math\'{e}matiques de Jussieu - Paris Rive Gauche, Sorbonne Universit\'{e}\\
4 Pl. Jussieu\\
75005 Paris\\
France} 
\email{yukako.kezuka@imj-prg.fr}

\begin{document}

\begin{abstract} 
Given any cube-free integer $\lambda>0$, we study the $3$-adic valuation of the algebraic part of the central $L$-value of the elliptic curve
$$X^3+Y^3=\lambda Z^3.$$
We give a lower bound in terms of the number of distinct prime factors of $\lambda$, which, in the case $3$ divides $\lambda$, also depends on the power of $3$ in $\lambda$. This extends an earlier result of the author in which it was assumed that $3$ is coprime to $\lambda$. We also study the $3$-part of the Tate--Shafarevich group for these curves and show that the lower bound is as expected from the conjecture of Birch and Swinnerton-Dyer, taking into account also the growth of the Tate--Shafarevich group.
\end{abstract}

\maketitle

\section{Introduction and main result}

Let $\lambda$ be a non-zero integer. We define the curve $C_\lambda/\Q$ by the equation
$$C_\lambda: X^3+Y^3=\lambda Z^3.$$
This is an elliptic curve, and for a cube-free $\lambda$, $C_\lambda$ is a twist of $C_1$ by the cubic extension $\Q(\sqrt[3]{\lambda})/\Q$ obtained by adjoining the real cube root of $\lambda$ to $\Q$. It has complex multiplication by the ring of integers $\O$ of the imaginary quadratic field $K=\Q(\sqrt{-3})$.
A classical Weierstrass equation for $C_\lambda$ is given by
$$E_\lambda: y^2z=4x^3-3^3\lambda^2z^3,$$
which is obtained by the change of variables $x= 3\lambda Z$, $y=3^2\lambda(Y-X)$, $z=X+Y$. This family of elliptic curves has been studied by a great many authors, including Selmer \cite{selmer}, Cassels \cite{cassels}, Rodriguez Villegas and Zagier \cite{RZ}. The study of these equations comes down to the classical question in Diophantine analysis: which integers $\lambda$ are sums of two rational cubes? Indeed, in the case $\lambda=p$ is a prime number, there is a conjecture of from 1879 by Sylvester \cite{sylvester} which says that if $p$ is congruent to $4,7$ or $8$ modulo $9$, then $p$ is a sum of two rational cubes.

We write $\Omega_\lambda$ for the least positive real period of the N\'{e}ron differential of a minimal generalised Weierstrass equation for $C_\lambda$. Then we define the algebraic part of $L(C_\lambda,1)$ to be
$$L^{\mathrm{(alg)}}(C_\lambda,1)=\frac{L(C_\lambda,1)}{\Omega_\lambda}.$$

In \cite{yukatam}, we studied the $3$-adic valuation of $L^{\mathrm{(alg)}}(C_\lambda,1)$ in the case when $(\lambda,3)=1$. We crucially used that $3\nmid \lambda$, because $3$ divides the conductor of the base curve $E_1$ which is $(27)$. The aim of this paper is to extend the lower bounds for the $3$-adic valuation of $L^{\mathrm{(alg)}}(C_\lambda,1)$ obtained in \cite{yukatam} to the case when $3\mid \lambda$. This is achieved by working with the Weierstrass $\wp$-function whose real period $\Omega$ is given by $\frac{\Omega}{\sqrt{3}}=\Omega_1$, satisfying the classical Weierstrass equation
$$\wp'^2(u,\mc{L})=4 \wp^3(u,\mc{L})-1$$
where $\mc{L}=\Omega \O$ denotes the period lattice. We explicitly compute the $3$-adic valuation of expressions involving $\wp(u,\mc{L})$ following the method of Stephens \cite{stephens}, where it was used to obtain integrality of $L^{(\mathrm{alg})}(C_\lambda,1)$, even when $3\mid \lambda$. Stephen's method, in turn, is based on the idea of Birch and Swinnerton-Dyer used in their proof of integrality of the algebraic part of the central $L$-value  of quadratic twists of the curve $y^2=4x^3-4x$ in \cite{BSD}. In this paper, we show in addition to integrality of $L^{\mathrm{(alg)}}(C_\lambda,1)$ that it is also divisible by an arbitrarily large power of $3$ as one increases the number of distinct prime divisors of $\lambda$. In order to do this, we introduce an ``averaged  $L$-value'' $\Phi_\lambda$, summing over the $L$-values of all cubic twists of $C_\lambda$ with the same (sub)set of prime divisors, while fixing the exponent $e(\lambda)$ of $3$ dividing $\lambda$. It is important that $e(\lambda)$ is fixed, as it turns out that the $3$-adic valuation of $L^{(\mathrm{alg})}(C_\lambda,1)$ depends on $e(\lambda)$ (see Theorem \ref{main3}).

From now on, we write $\lambda=3^{e(\lambda)}D$ with $(3,D)=1$. We may assume that $\lambda$ is a cube-free positive integer so that $e(\lambda)\in \{1,2\}$, since the case $3\nmid \lambda$ was treated in \cite{yukatam}. Let $p_1,\ldots , p_{n}$ be the distinct prime divisors dividing $\lambda$, and without loss of generality we set $p_1=3$. Given $\alpha=(\alpha_2,\ldots , \alpha_n)\in \Z^{n-1}$, we define
$$D_\alpha=p_2^{\alpha_2}\cdots p_n^{\alpha_n}$$
and
$$\lambda_\alpha=3^{e(\lambda)}D_\alpha.$$
Given $\alpha, \beta \in \Z^{n-1}$, the resulting curves $C_{\lambda_\alpha}$ and $C_{\lambda_{\beta}}$ are isomorphic over $\Q$ if $\alpha_i\equiv \beta_i \bmod 3$ for all $i\in \{2,\ldots n\}$. Thus, we may treat $\alpha$ as an element of $(\Z/3\Z)^{n-1}$.

Suppose $L^{(\mathrm{alg})}(C_\lambda,1)\neq 0$. Then the conjecture of Birch and Swinnerton-Dyer predicts:
\begin{conj}\label{BSD} Let $\lambda$ be a non-zero, cube-free integer such that $L^{(\mathrm{alg})}(C_\lambda,1)\neq 0$. Then
$$L^{(\mathrm{alg})}(C_\lambda,1)=\frac{\#(\Sha(C_\lambda/\Q))\prod_{q} c_q}{\#(C_\lambda(\Q)_{\mathrm{tor}})^2},$$
where $\Sha(C_\lambda/\Q)$ is the Tate--Shafarevich group of $C_\lambda/\Q$, $c_q$ denotes the Tamagawa factor of $C_\lambda$ at the prime $q$, and $C_\lambda(\Q)_{\mathrm{tor}}$ denotes the torsion part of the Mordell--Weil group $C_\lambda(\Q)$.
\end{conj}

\begin{remark}\label{remarkGZK}
Since $C_\lambda$ admits complex multiplication by $K$ with class number $1$, a theorem of Coates and Wiles \cite{CW} shows that if $L(C_\lambda,1)\neq 0$ then $C_\lambda(\Q)$ is a finite group. Shimura proved in \cite{Shimura} that any elliptic curve $E/\Q$ with complex multiplication is modular. More generally, given any elliptic curve $E/\Q$, we know that $E$ is modular by the work of Wiles \cite{wil}, Taylor--Wiles \cite{TW} and Breuil--Conrad--Diamond--Taylor \cite{BCDT}. Thus, by a theorem of Gross, Zagier \cite{GZ} and Kolyvagin \cite{kol}, we know that if the analytic rank $\ord_{s=1}L(E,s)$ is less than or equal to $1$, then it is equal to the algebraic rank. In addition, Kolyvagin showed that $\Sha(E/\Q)$ is finite in this case. 
\end{remark}
\begin{remark}
When $\lambda=2$, we have
$C_\lambda(\Q)_{\mathrm{tor}}=\Z/2\Z$ (note that in \cite[X.10.19 (d)]{sil} there is a misprint that $C_\lambda(\Q)_{\mathrm{tor}}=\Z/3\Z$). Otherwise, we have $C_\lambda(\Q)_{\mathrm{tor}}$ is trivial for a cube-free $\lambda$. 
\end{remark}
\begin{remark}The $p$-part of Conjecture \ref{BSD} was obtained by Rubin \cite{rubin} using Iwasawa theory for $p\neq 2$ or $3$. In this paper, we contribute to the case $p=3$.
\end{remark}
Let $n=n(\lambda)$ be the number of distinct prime divisors of $\lambda$, and let $s(\lambda)$ (resp. $r(\lambda)$) denote the number of distinct prime divisors of $\lambda$ which are split (resp. inert) in $K$. Since we assume $3\mid \lambda$ and $3$ is the only prime which ramifies in $K$, we have
$$n(\lambda)=s(\lambda)+r(\lambda)+1.$$
A computation of Tamagawa factors using Tate's algorithm \cite{tate} gives $c_q=1$ if $q$ is a prime of good reduction,
$$
c_3= \left\{ \begin{array}{ll}
2 &\mbox{if $\lambda\equiv 2, 7 \bmod 9$} \\
3&\mbox{if $\lambda \equiv 1, 8 \bmod 9$} \\
1  & \mbox{$\lambda \equiv 0, 3, 4,  5, 6 \bmod 9$,}
       \end{array} \right.
$$
and for a prime $q$ of bad reduction,
$$
c_q= \left\{ \begin{array}{ll}
3 &\mbox{if $q\equiv 1 \bmod 3$} \\
1&\mbox{if $q \equiv 2 \bmod 3$.}
       \end{array} \right.
$$
In particular, if $3\mid \lambda$, we have
$$\prod_q c_q=3^{s(\lambda)}.$$
\begin{remark} In \cite[Lemma 3]{stephens} it is stated that $c_3=3$ for $\lambda\equiv 0\bmod 9$. This should be corrected to $c_3=1$.
\end{remark}
Satg\'{e}  gave a $3$-descent argument on $C_\lambda$ by identifying an element in the Selmer group of a $3$-isogeny to an element of $\Hom(\Gal(\ovl{\Q}/K),\Z/3\Z)$ whose kernel fixes the Galois closure $N$ of a pure cubic field of the form $\Q(\sqrt[3]{D_\alpha})$, satisfying certain local conditions. In particular, $N$ completed at a prime $p\neq 3$ dividing $\lambda$ must be contained in $\Q_p(\sqrt{-3},\sqrt[3]{\lambda^2})$ (see \cite[Proposition 1.7, Th\'{e}or\`{e}me 1.14]{satge}). This is satisfied automatically in the case $p$ is inert in $K$. This was used to show that for any cube-free positive integer $\lambda$, we have (see \cite[Theorem 3.1]{yukatam}):
\begin{equation}\label{Sha3}\ord_3\left(\#\Sha(C_\lambda)\right)\geq r(\lambda)-t(\lambda)-1-\rank(C_\lambda),
\end{equation}
where $\ord_3$ denotes the normalised $3$-adic valution such that $\ord_3(3)=1$ and $t(\lambda)=1$ if $\lambda\equiv \pm 1 \bmod 9$, $t(\lambda)=-1$ if $\ord_3(\lambda)=1$ and $t(\lambda)=0$ otherwise. In Theorem \ref{3descent}, we improve this result and show that equality holds in certain special cases. Furthermore, it was already known by Cassels that $\Sha(C_N)$ can be arbitrarily large. Indeed, by directly working with principal homogeneous spaces and choosing $\lambda$ so that these equations are everywhere locally solvable but not globally (and thus, by definition, an element of $\Sha(C_\lambda)$), he was able to show that given any positive integer $n$, there exists $\lambda$ with $n+3^n$ distinct prime factors, all congruent to $8$ modulo $9$, such that $\ord_3\left(\#(\Sha(C_\lambda))\right)\geq n$. Unfortunately, there is no direct way of choosing such a $\lambda$. Theorem \ref{3descent} (3) provides another perspective and shows that in this case we have an equality:
$$\ord_3(\#(\Sha(C_\lambda))= n+3^n-\rank(C_\lambda).$$

Note that when $\ord_{s=1}(L(C_\lambda,s))\leq 1$, we know by Remark \ref{remarkGZK} that analytic and the algebraic ranks coincide and that $\Sha(C_\lambda/\Q)$ is finite. In general, $\Sha(C_\lambda/\Q)$ is conjectured to be finite, and when it is, we know by a theorem of Cassels \cite{cassels62} that its order is a perfect square. We check that the global root number $\epsilon(C_\lambda/\Q)$ has the same parity as $r(\lambda)-t(\lambda)-1$ (see Proposition \ref{proproot}), which is consistent with the fact that we expect the left hand side of \eqref{Sha3} to be an even integer. Now, in the case $e(\lambda)\in \{1,2\}$, we have $e(\lambda)=t(\lambda)+2$ and thus we expect from the exact formula of the Birch--Swinnerton-Dyer conjecture:
\begin{align}\label{eqmain3}
\nonumber \ord_3\left(L^{(\mathrm{alg})}(C_\lambda,1)\right)&\geq s(\lambda)+r(\lambda)-e(\lambda)+1\\
&=n(\lambda)-e(\lambda)
\end{align}
where the inequality holds automatically when $L(C_\lambda,1)=0$, since in this case $\ord_3\left(L^{(\mathrm{alg})}(C_\lambda,1)\right)=\infty$. The main result of the paper is that this inequality \eqref{eqmain3} predicted by the Birch--Swinnerton-Dyer conjecture indeed holds (see Theorem \ref{main3}). Thus, combined with \cite[Theorem 1.3]{yukatam}, we obtain:
\begin{thm}\label{main} Let $\lambda>1$ be any cube-free integer. Then
$$
\ord_3(L^{(\mathrm{alg})}(C_\lambda,1))\geq \left\{ \begin{array}{ll}
n(\lambda)&\mbox{if the prime factors of $\lambda$ split in $K$} \\
n(\lambda)-2&\mbox{if $9\mid \lambda$} \\
n(\lambda)-1 & \mbox{otherwise.}
       \end{array} \right.
$$
\end{thm}
Thus, the $3$-adic valuation of $\ord_3(L^{(\mathrm{alg})}(C_\lambda,1))$ grows with the number of distinct prime divisors of $\lambda$. Furthermore, the exact formula of Birch and Swinnerton-Dyer conjecture suggests that the growth of $\ord_3(L^{(\mathrm{alg})}(C_\lambda,1))$ with $s(\lambda)$ comes from the growth of $\ord_3\left(\prod_q c_q\right)$, and that the growth of $\ord_3(L^{(\mathrm{alg})}(C_\lambda,1))$ with $r(\lambda)$ comes from the growth of $\Sha(C_\lambda)[3]$. In particular, we note that Theorem \ref{main} is strictly stronger than the bound we get from Conjecture \ref{BSD} assuming that $\#(\Sha(C_\lambda))$ is an integer, as it also reflects the fact that $\Sha(C_\lambda)[3]$ grows with the number of distinct prime divisors of $\lambda$ which are inert in $K$.

The numerical examples listed in this paper were obtained using PARI/GP \cite{pari} and Magma \cite{magma}.

\section{Averaging of the algebraic part of the central $L$-values}

For ease of computation, we let $\Omega=3.059908...$ denote the real period of the classical Weierstrass equation $y^2=4x^3-1$, and set $\mc{L}=\Omega\O$. Let $\Omega_1=1.76663875...$  denote the real period of the minimal model $y^2+y=x^3-7$ of $E_1$. Then $\Omega_1=\frac{\Omega}{\sqrt{3}}$. In general, we set $\Omega_\lambda=\frac{\Omega}{\sqrt{3}{\sqrt[3]{\lambda}}}$ when $\ord_3(\lambda)\leq 1$ and $\Omega_\lambda=\frac{3\Omega}{\sqrt{3}\sqrt[3]{\lambda}}=\frac{\sqrt{3}\Omega}{\sqrt[3]{\lambda}}$ when $\ord_3(\lambda)=2$, and we write $\mc{L}_\lambda=\Omega_\lambda\O$. Then the algebraic part of $L(C_\lambda,1)$ is given by
$$L^{\mathrm{(alg)}}(C_\lambda,1)=\frac{L(C_\lambda,1)}{\Omega_\lambda}.$$

Let $\psi_\lambda$ denote the Grossencharacter of $C_\lambda/K$, and let $\ovl{\psi}_\lambda$ denote its complex conjugate. Then by a theorem of Deuring (see, for example, \cite[Theorem 45]{coates1}), we have
$$L(C_\lambda,s)=L(\ovl{\psi}_\lambda,s).$$
Thus,
$$
\ord_3\left(\frac{L(\ovl{\psi}_{\lambda},1)}{\Omega}\right)= \left\{ \begin{array}{ll}
\ord_3\left(L^{(\mathrm{alg})}(C_\lambda,1)\right)+\ord_3\left(\frac{1}{\sqrt{3}\sqrt[3]{\lambda}}\right)  &\mbox{if $e(\lambda)=1$} \\
\ord_3\left(L^{(\mathrm{alg})}(C_\lambda,1)\right)+\ord_3\left(\frac{\sqrt{3}}{\sqrt[3]{\lambda}}\right)  & \mbox{if $e(\lambda)=2$,}
       \end{array} \right.
$$
so that
       \begin{align}\label{phialg}
\ord_3\left(\frac{L(\ovl{\psi}_{\lambda},1)}{\Omega}\right)= \left\{ \begin{array}{ll}
\ord_3\left(L^{(\mathrm{alg})}(C_\lambda,1)\right) -\frac{5}{6} &\mbox{if $e(\lambda)=1$} \\
\ord_3\left(L^{(\mathrm{alg})}(C_\lambda,1)\right) -\frac{1}{6}  & \mbox{if $e(\lambda)=2$.}
       \end{array} \right.
\end{align}

Let $\mf{f}$ denote the conductor of $\psi_{\lambda}$ where $\ord_3(\lambda)=1$ or $2$, and recall that $\lambda=3^{e(\lambda)}D$ so that $(3,D)=1$. Then by Tate's algorithm \cite{tate} and the relation

$$C(E_\lambda/\Q)=\mathrm{N}_{K/\Q}\mf{f}\cdot d_K$$
where $C(E_\lambda/\Q)$ denotes the conductor of $E_\lambda$ over $\Q$ and $d_K$ denotes the absolute value of the discriminant of $K/\Q$, we have
$$\mf{f}=9\Delta \O,$$
where $\Delta=\epsilon\cdot \mathrm{rad}(D)$, $\epsilon\in \{\pm 1\}$ is chosen so that $\Delta\equiv 1\bmod 3$ and $\mathrm{rad}(D)$ is the product of distinct rational primes dividing $D$. When $\Delta=-2$, we can compute that for any $e(\lambda)\in \{1,2\}$ and $\alpha_2\in \Z/3\Z$, we have $L^{\mathrm{alg}}(C_3,1)=L^{\mathrm{alg}}(C_{3^2 2},1)=L^{\mathrm{alg}}(C_{3^22^2},1)=1$ and $L^{\mathrm{alg}}(C_{3^{e(\lambda)2^{\alpha_2}}},1)=0$ otherwise. Thus, we may assume from now on that $|\Delta|>2$. 

For any prime $\mf{p}$ of $K$ coprime to $\mf{f}$, we have (see \cite[Chapter II, Example 10.6]{silverman})
$$\psi_{\lambda}(\mf{p})=\ovl{\left(\frac{\lambda}{\mf{p}}\right)}_3 \pi,$$
where $\pi$ is a generator of $\mf{p}$ such that $\pi\equiv 1\bmod 3$.

Let $\mc{B}$ be a set of integral ideals of $K$ prime to $\mf{f}$ such that the Artin symbols $\sigma_\mf{b}$ of $\mf{b}$ in $K(\mf{f})/K$ give precisely the Galois group $\Gal(K(\mf{f})/K)$ as $\mf{b}$ runs through the elements in $\mc{B}$. Recall that $n=n(\lambda)$ is the number of distinct prime divisors $p_1=3, p_2,\ldots , p_n$ of $\lambda$, and set $S=\{p_1,\ldots , p_n\}$. Given $\alpha=(\alpha_2,\ldots ,\alpha_{n})\in (\Z/3\Z)^{n-1}$, we set
$$D_\alpha=p_2^{\alpha_2}\cdots p_n^{\alpha_n},$$ 
and put
$$\lambda_\alpha=3^{e(\lambda)}D_\alpha.$$ 
For any $\lambda_\alpha\mid \lambda$, we write $\mf{f}_\alpha$ for the conductor of the Grossencharacter $\psi_{\lambda_\alpha}$ associated to $C_{\lambda_\alpha}/K$. Then $\mf{f}_\alpha$ divides $\mf{f}$, and thus we have $K(C_{\lambda_\alpha,\mf{f}})=K(\mf{f})$. 
Then the (usually) imprimitive $L$-function is given by
\begin{align}\label{imprimitive}L_S(\ovl{\psi}_{\lambda_\alpha},1)&=\prod_{v\nmid \mf{f}}\left(1-\frac{\ovl{\psi}_{\lambda_\alpha}(v)}{(\mathrm{N}v)^s}\right)^{-1}\\
&=\sum_{(\mf{a},\mf{f})=1}\frac{\ovl{\psi}_{\lambda_\alpha}(\mf{a})}{(\mathrm{N}\mf{a})^s}\\
&=\sum_{(\mf{a},\mf{f})=1}\frac{\left(\frac{\lambda_\alpha}{\mf{a}}\right)_3\ovl{\alpha}}{(\mathrm{N}\mf{a})^s},
\end{align}
where the sum runs over integral ideals of $\O$ prime to $\mf{f}$, and $\alpha$ is a generator of $\mf{a}$ congruent to $1$ modulo $3$.
We will often omit the $S$ from the notation whenever the $L$-function is primitive, that is, the primes dividing $\mf{f}_\alpha$ coincide with those dividing $\mf{f}$. Note that we have $3\mid \mf{f}_\alpha$ for any $\alpha\in (\Z/3\Z)^{n-1}$. By class field theory, the Artin map gives an isomorphism
$$\left(\O/9\Delta\right)^\times/\widetilde{\bold{\mu}}_6\xto{\sim}\Gal(K(\mf{f})/K).$$
Note that $|\O/9\O|^\times=2 \cdot 3^3$ and $|\left(\O/9\O\right)/\widetilde{\bold{\mu}}_6|^\times=9$. 
We may take a set of representatives $\tilde{\mc{C}}$ of $(\O/\Delta\O)^\times$ such that $c\in \tilde{\mc{C}}$ implies $-c\in \tilde{\mc{C}}$. This is clearly possible in the case where $\Delta$ is odd. If $\Delta$ is even, it is of the form $\Delta=2\Delta'$ where $\Delta'$ is odd. Thus if $c\equiv -c\bmod 2\Delta'$ for $c\in \left(\O/2\Delta ' \O\right)^\times$, it would imply $c\equiv 0\bmod \Delta'$, which is absurd. We can define a set of integral ideals of $K$ by
$$\mc{B}=\{\mf{b}=(9c+\Delta): c \in \mc{C}\},$$
where $\mc{C}$ consists of half of the set of representatives $\tilde{\mc{C}}$ of $(\O/\Delta\O)^\times$ such that $c\in \mc{C}$ implies $-c\not\in \mc{C}$.  Note that $9c+\Delta\equiv 1\bmod 3$, so that 
$$\left(\frac{p}{\mf{b}}\right)_3=\left(\frac{9c}{p}\right)_3=\left(\frac{c}{p}\right)_3$$
for any $\mf{b}=(9c+\Delta)\in \mc{B}$ and any prime divisor $p$ of $D$.

Thus, when $e(\lambda)=1$, given any $\alpha\in (\Z/3\Z)^{n-1}$, we obtain from \cite[(3.8), Lemma 3]{stephens} an expression of the form
\begin{align}\label{K1}\frac{L_S(\psi_{\lambda_\alpha},1)}{\Omega}=\frac{1}{3^\frac{5}{6}\Delta}\sum_{\mf{b}\in \mc{B}}\left(\frac{D_\alpha}{\mf{b}}\right)_3 E_1(c)
\end{align}
where $\mf{b}=(9c+\Delta)$ and $E_1(c)$ is expressed in terms of $\wp(\frac{9c\Omega}{\Delta},\mc{L})$. In order to study the $3$-adic valuation of $E_1(c)$, Stephens uses the following lemma (see \cite[Lemma 1, Lemma 2]{stephens}): 

\begin{lem}\label{ste} Let $r>0$ be a rational integer, and let $\beta$  and $\gamma$ be elements of $\O$ such that $\beta$ is prime to $\sqrt{-3}$ and $\gamma$  is prime to $\Delta$. Define $s=\frac{1}{2}(1-3^{1-r})$ and
$$
\varphi(\Delta)= \left\{ \begin{array}{ll}
\Delta^{2(\Delta^2-1)}&\mbox{if $\Delta$ is prime} \\
1   & \mbox{otherwise.}
       \end{array} \right.
$$
Then $\varphi(\Delta)\wp\left(\frac{\gamma\Omega}{\Delta}, \mc{L}\right)$ and $3^{-s}\varphi(\Delta)\left(\wp\left(\frac{\delta\Omega}{(\sqrt{-3})^r},\mc{L}\right)-\wp\left(\frac{\gamma\Omega}{\Delta},\mc{L}\right)\right)$ are algebraic units.
\end{lem}
This idea is similar to that of \cite[\S 4]{BSD} for the curve $y^2=4x^3-4x$, and the proof uses induction on $r$, but also uses that $\wp(\frac{\Omega}{\sqrt{-3}},\mc{L})=0$, and $\wp(\sqrt{-3}u,\mc{L})=\frac{1-\wp^3(u,\mc{L})}{3\wp^2(u,\mc{L})}$ for the Weierstrass $\wp$-function satisfying
$$\wp'^2(u,\mc{L})=4\wp^2(u,\mc{L})-1.$$ Lemma \ref{ste} is applied to the explicit expression of $E_1(c)$ to give (see the proof of \cite[Lemma 3]{stephens}) $\ord_3(E_1(c))\geq -\frac{1}{3}$.
Similarly, when $e(\lambda)=2$, we obtain from \cite[(3.9), Lemma 3]{stephens} an expression of the form
\begin{align}\label{K2}\frac{L_S(\psi_{\lambda_\alpha},1)}{\Omega}=\frac{3^\frac{5}{6}}{\Delta}\sum_{\mf{b}\in \mc{B}}\left(\frac{D_\alpha}{\mf{b}}\right)_3 E_2(c)
\end{align}
where $E_2(c)$ is another expression in terms of $\wp(\frac{9c\Omega}{\Delta},\mc{L})$ and $\ord_3(E_2(c))\geq -\frac{5}{3}$, again by Lemma \ref{ste}. As mentioned in the introduction, this expression is used in \cite{stephens} to obtain integrality of $L^{(\mathrm{alg})}(C_\lambda,1)$.  We now introduce an ``averaged $L$-value'' to study 
$\ord_3\left(L^{(\mathrm{alg})}(C_\lambda,1)\right)$ further. 

Define 
$$\Phi_\lambda=\sum_{\alpha\in (\Z/3\Z)^{n-1}} \frac{L_S(\ovl{\psi}_{3^{e(\lambda)}D_\alpha},1)}{\Omega}.$$

Let 
$$\tilde{V}=\{c \in \tilde{\mc{C}}: \left(\frac{p}{\mf{b}}\right)_3=1 \text{ for all $p\mid \Delta$, where $\mf{b}=(9c+\Delta)$}\}.$$
Since $\left(\frac{-1}{\pi}\right)_3=1$ for any prime $\pi$, we have have $c\in \tilde{V}$ then $-c\in \tilde{V}$. By the definition of the cubic residue symbol, we have $\sqrt[3]{p}\in K(\mf{f})$ for any $p\mid \Delta$. We may identify $\tilde{V}$ with the Galois group $\Gal\left(K(\mf{f})/K_9(\sqrt[3]{p_2},\ldots , \sqrt[3]{p_n})\right)$, where $K_9$ denotes the ray class field of $K$ modulo $9\O$. Similarly to how we obtained $\mc{C}$ from $\tilde{\mc{C}}$, we let
$$V=\{c \in \mc{C}: \left(\frac{p}{\mf{b}}\right)_3=1 \text{ for all $p\mid \Delta$, where $\mf{b}=(9c+\Delta)$}\}$$
so that $\#(V)=\frac{1}{2}\#(\tilde{V})$ and $c\in V$ implies $-c\not\in V$.

\begin{lem}\label{ordPhi} Let $\lambda>1$ be a cube-free integer such that $\lambda=3^{e(\lambda)}D$ with $(3,D)=1$ and $\ord_3(\lambda)=e(\lambda)\in \{1,2\}$. Then
$$
\ord_3\left(\Phi_{\lambda}\right)\geq \left\{ \begin{array}{ll}
n(\lambda)-\frac{13}{6}&\mbox{if $e(\lambda)=1$} \\
n(\lambda)-\frac{11}{6}   & \mbox{if $e(\lambda)=2$.}
       \end{array} \right.
$$
\end{lem}
\begin{proof}

Suppose that $e(\lambda)=1$, and given $\alpha\in (\Z/3\Z)^{n-1}$ write $\lambda_\alpha=3D_\alpha$ with $(3,D_\alpha)=1$. By \eqref{K1}, we have
\begin{align*}\frac{L_S(\psi_{\lambda_\alpha},1)}{\Omega}=\frac{1}{3^\frac{5}{6}\Delta}\sum_{\mf{b}\in \mc{B}}\left(\frac{D_\alpha}{\mf{b}}\right)_3 E_1(c)
\end{align*}
where $\ord_3(E_1(c))\geq -\frac{1}{3}$. Now, sending $\alpha$ to $\left(\frac{D_\alpha}{\mf{b}}\right)_3$ gives an irreducible character of the abelian group $(\Z/3\Z)^{n-1}$. Thus, $\frac{1}{|(\Z/3\Z)^{n-1}|}\sum_{\alpha \in (\Z/3\Z)^{n-1}}\left(\frac{D_\alpha}{\mf{b}}\right)_3$ is its inner product with the trivial character, and is equal to $1$ precisely when $\mf{b}=(9c+\Delta)$ with $c\in V$ and is equal to $0$ otherwise. It follows that 
\begin{align*}\Phi_{\lambda}&=\frac{1}{3^\frac{5}{6}\Delta}\sum_{\alpha\in (\Z/3\Z)^{n-1}}\sum_{c\in \mc{C}}\left(\frac{D_\alpha}{\mf{b}}\right)_3 E_1(c)\\
&=\frac{1}{3^\frac{5}{6}\Delta}\sum_{c\in V}3^{n-1} E_1(c).
\end{align*}
Thus $\ord_3(\Phi_{\lambda})\geq-\frac{5}{6}+n-1-\frac{1}{3}=n-\frac{13}{6}$.
Next, let $e(\lambda)=2$, and write $\lambda_\alpha=3^2D_\alpha$ with $(3,D_\alpha)=1$. By \eqref{K2}, we have
\begin{align*}\frac{L_S(\psi_{\lambda_\alpha},1)}{\Omega}=\frac{3^\frac{5}{6}}{\Delta}\sum_{\mf{b}\in \mc{B}}\left(\frac{D_\alpha}{\mf{b}}\right)_3 E_2(c)
\end{align*}
where $\ord_3(E_2(c))\geq -\frac{5}{3}$. Thus 

\begin{align*}\Phi_{\lambda}&=\frac{3^\frac{5}{6}}{\Delta}\sum_{\alpha\in (\Z/3\Z)^{n-1}}\sum_{\mf{b}\in \mc{B}}\left(\frac{D_\alpha}{\mf{b}}\right)_3 E_2(c)\\
&=\frac{3^\frac{5}{6}}{\Delta}\sum_{c\in V}3^{n-1} E_2(c).
\end{align*}
Thus $\ord_3(\Phi_{\lambda})\geq \frac{5}{6}+n-1-\frac{5}{3}=n-\frac{11}{6}$.
\end{proof}
Following the argument of \cite[Corollary 2.3]{yukatam}, we obtain 
\begin{cor}\label{ordPhichi} Let $\lambda>1$ be a cube-free integer such that $\lambda=3^{e(\lambda)}D$ with $(3,D)=1$ and $\ord_3(\lambda)=e(\lambda)\in \{1,2\}$. For any character $\chi$ of $(\Z/3\Z)^{n-1}$, define 
$$\Phi_{\lambda}^{(\chi)}= \sum_{\alpha\in (\Z/3\Z)^{n-1}}\chi(\alpha)\frac{L_S(\ovl{\psi}_{3^{e(\lambda)}D_\alpha},1)}{\Omega}$$
Then we have
$$
\ord_3\left(\Phi_{\lambda}^{(\chi)}\right)\geq \left\{ \begin{array}{ll}
n(\lambda)-\frac{13}{6}&\mbox{if $e(\lambda)=1$} \\
n(\lambda)-\frac{11}{6}   & \mbox{if $e(\lambda)=2$.}
       \end{array} \right.
$$
\end{cor}

We are now ready to prove the main result of this paper. 
\begin{thm}\label{main3} Let $\lambda> 1$ be a cube-free integer such that $\ord_3(\lambda)=e(\lambda)\in \{1,2\}$. Then
$$
\ord_3(L^{(\mathrm{alg})}(C_\lambda,1))\geq n(\lambda)-e(\lambda).
$$
\end{thm}

\begin{proof}
We have $L^{(\mathrm{alg})}(C_1,1)=\frac{1}{3}$, $L^{(\mathrm{alg})}(C_3,1)=1$ and $L^{(\mathrm{alg})}(C_9,1)=0$. It suffices to show by induction on $n(\lambda)$ that $\ord_3\left(\frac{L(\ovl{\psi}_{\lambda},1)}{\Omega}\right)\geq n(\lambda) - \frac{11}{6}$ when $e(\lambda)=1$ and $\ord_3\left(\frac{L(\ovl{\psi}_{\lambda},1)}{\Omega}\right)\geq n(\lambda) - \frac{13}{6}$ when $e(\lambda)=2$. We first treat the case $e(\lambda)>0$ and $n(\lambda)=2$. We may assume $\lambda=3^{e(\lambda)}p$, and the case $\lambda=3^{e(\lambda)}p^2$ is similar. Given a prime number $p$, the root number satisfies (see Proposition \ref{proproot}) $\epsilon(C_{3p}/\Q)=\epsilon(C_{3p^2}/\Q)=\left(\frac{-3}{p}\right)$. Thus, we may assume $p$ is split, say $p\O=(\pi)(\ovl{\pi})$ with $\pi\equiv 1\bmod 3\O$. Then we have
\begin{align*}\Phi_{3p}&=\frac{L_S(\ovl{\psi}_3,1)}{\Omega}+\frac{L_S(\ovl{\psi}_{3p},1)}{\Omega}+\frac{L_S(\ovl{\psi}_{3p^2},1)}{\Omega}\\
&=\left(1-\frac{\ovl{\psi}_3((\pi))}{\mathrm{N}\pi}\right)\left(1-\frac{\ovl{\psi}_3((\ovl{\pi}))}{\mathrm{N}\ovl{\pi}}\right)\frac{L(\ovl{\psi}_3,1)}{\Omega}+\frac{L(\ovl{\psi}_{3p},1)}{\Omega}+\frac{L(\ovl{\psi}_{3p^2},1)}{\Omega}\\
&=\left(\frac{\pi -\left(\frac{3}{\pi}\right)_3}{\pi}\right)\left(\frac{\ovl{\pi}-\left(\frac{3}{\ovl{\pi}}\right)_3}{\ovl{\pi}}\right)\frac{L(\ovl{\psi}_3,1)}{\Omega}+\frac{L(\ovl{\psi}_{3p},1)}{\Omega}+\frac{L(\ovl{\psi}_{3p^2},1)}{\Omega}.
\end{align*}
We have $\ord_3\left(\left(\frac{\pi -\left(\frac{3}{\pi}\right)_3}{\pi}\right)\left(\frac{\ovl{\pi}-\left(\frac{3}{\ovl{\pi}}\right)_3}{\ovl{\pi}}\right)\frac{L(\ovl{\psi}_3,1)}{\Omega}\right)\geq \frac{1}{2}+\frac{1}{2}+0=1$. We write $\chi$ for the character of $(\Z/3\Z)$ sending $1$ to $\omega$. Then 
$$\Phi_{3p}^{(\chi)}-\omega^2\Phi_{3p}=(1-\omega)\frac{L_S(\ovl{\psi}_{3},1)}{\Omega}+(\omega-\omega^2)\frac{L(\ovl{\psi}_{3p},1)}{\Omega}.$$
By Lemma \ref{ordPhi} and Corollary \ref{ordPhichi}, $\ord_3(\Phi_{3p}^{(\chi)}-\omega^2\Phi_{3p})\geq 2-\frac{13}{6}=-\frac{1}{6}$. Since $\ord_3(\omega-\omega^2)=\frac{1}{2}$, it follows that $\ord_3\left(\frac{L(\ovl{\psi}_{3p},1)}{\Omega}\right)\geq -\frac{1}{6}-\frac{1}{2}=-\frac{2}{3}$.
Thus, $\ord_3\left(L^{(\mathrm{alg})}(C_{3p},1)\right)=-\frac{2}{3}+\frac{5}{6}=\frac{1}{6}$, and since $L^{\mathrm{(alg)}}(C_{3p},1)\in \Z$, it follows that $\ord_3\left(L^{(\mathrm{alg})}(C_{3p},1)\right)\geq 1=n(3p)-e(3p)$, as required. Next, we consider the case $n(\lambda)=e(\lambda)=2$. Given a prime number $p$, the root number satisfies $\epsilon(C_{9p}/\Q)=\epsilon(C_{9p^2}/\Q)=-\left(\frac{-3}{p}\right)$. Thus, we may assume $p$ is inert. Then we have
\begin{align*}\Phi_{9p}&=\frac{L_S(\ovl{\psi}_9,1)}{\Omega}+\frac{L_S(\ovl{\psi}_{9p},1)}{\Omega}+\frac{L_S(\ovl{\psi}_{9p^2},1)}{\Omega}\\
&=\left(1-\frac{\ovl{\psi}_9((p))}{\mathrm{N}p}\right)\frac{L(\ovl{\psi}_9,1)}{\Omega}+\frac{L(\ovl{\psi}_{9p},1)}{\Omega}+\frac{L(\ovl{\psi}_{9p^2},1)}{\Omega}\\
&=\frac{L(\ovl{\psi}_{9p},1)}{\Omega}+\frac{L(\ovl{\psi}_{9p^2},1)}{\Omega},
\end{align*}
since we have $L(\ovl{\psi}_9,1)=0$. Similarly, we have $\ord_3(\Phi_{9p}^{(\chi)})\geq \frac{1}{6}$ for any character $\chi$ of $(\Z/3\Z)$, so that $\ord_3\left(\frac{L(\ovl{\psi}_{9p},1)}{\Omega}\right)\geq \frac{1}{6}-\frac{1}{2}=-\frac{1}{3}$. It follows that $\ord_3\left(L^{(\mathrm{alg})}(C_{9p},1)\right)=-\frac{1}{3}+\frac{1}{6}=-\frac{1}{6}$. Since $L^{(\mathrm{alg})}(C_{9p},1)\in \Z$ by \cite{stephens}, it follows that $\ord_3\left(L^{(\mathrm{alg})}(C_{9p},1)\right)=0=n(9p)-e(9p)$, as required.

Now suppose that $n(\lambda)>2$ and $\lambda=3D$ with $(3,D)=1$. Assume that given any $\beta\in (\Z/3\Z)^{n-1}$ with $n(D_\beta)<n(D)$, we have $\ord_3\left(\frac{L(\ovl{\psi}_{3D_\beta},1)}{\Omega}\right)\geq n(3D_\beta) - \frac{11}{6}$. 
Then

$$\Phi_{\lambda}=\frac{L_S(\ovl{\psi}_3,1)}{\Omega}+\sum_{n(D_\beta)<n(\lambda)-1}\frac{L_S(\ovl{\psi}_{3D_\beta},1)}{\Omega}+\sum_{n(D_\alpha)=n(\lambda)-1}\frac{L_S(\ovl{\psi}_{3D_\alpha},1)}{\Omega}.$$
The first term on the right hand side is
$$\frac{L_S(\ovl{\psi}_3,1)}{\Omega}=\prod_{p\mid D \text{ inert}}\left(\frac{p^2+p}{p}\right)\prod_{p\mid D \text{ split}} \left(\frac{\pi-\left(\frac{3}{\pi}\right)_3}{\pi}\right)\left(\frac{\ovl{\pi}-\left(\frac{3}{\ovl{\pi}}\right)_3}{\ovl{\pi}}\right)\frac{L(\ovl{\psi}_3,1)}{\Omega},$$
and thus $\ord_3\left(\frac{L_S(\ovl{\psi}_3,1)}{\Omega}\right)\geq r(\lambda)+s(\lambda)= n(\lambda)-1$. Next, given $\beta\in (\Z/3\Z)^{n-1}$  such that $n(D_\beta)<n(\lambda)-1$, we have
\begin{small}
\begin{align*}\frac{L_S(\ovl{\psi}_{3D_\beta},1)}{\Omega}&=\prod_{\substack{p\in S\backslash S_\beta\\\text{inert}}}\left(\frac{p^2+p}{p}\right) \prod_{\substack{p\in S\backslash S_\beta\\  \text{split} }}\left(\frac{\pi-\left(\frac{3D_\beta}{\pi}\right)_3}{\pi}\right)\left(\frac{\ovl{\pi}-\left(\frac{3D_\beta}{\ovl{\pi}}\right)_3}{\ovl{\pi}}\right)\frac{L(\ovl{\psi}_{3D_\beta},1)}{\Omega}.
\end{align*}
\end{small}
By the induction hypothesis, we have
\begin{align*}\ord_3\left(\frac{L_S(\ovl{\psi}_{3D_\beta},1)}{\Omega}\right)&=\left(r(\lambda)-r(3D_\beta)\right)+\frac{1}{2}\left(2s(\lambda)-2s(3D_\beta)\right)+n(3D_\beta) - \frac{11}{6}\\
&=r(\lambda)+s(\lambda)+1-\frac{11}{6}=n(\lambda)-\frac{11}{6}.
\end{align*}

Thus we have shown that $\ord_3\left(\sum_{\alpha}\frac{L(\ovl{\psi}_{3D_\alpha},1)}{\Omega}\right)\geq n(\lambda)-\frac{11}{6}.$
By the proof of \cite[Theorem 2.4]{yukatam}, this lower bound applies to the individual summands. In particular, $\ord_3\left(\frac{L(\ovl{\psi}_{3D},1)}{\Omega}\right)\geq n(\lambda)-\frac{11}{6}$. It follows that $\ord_3\left(L^{(\mathrm{alg})}(C_{3D},1)\right)\geq n(\lambda)-1=n(3D)-e(3D)$, as required. In the case where $\lambda=9D$ with $(3,D)=1$, the argument showing $\ord_3\left(L^{(\mathrm{alg})}(C_{9D},1)\right)\geq n(\lambda)-2=n(9D)-e(9D)$ is identical, with $9$ replacing $3$ and $\frac{13}{6}$ replacing $\frac{11}{6}$.

\end{proof}

\section{On the $3$-part of the Tate--Shafarevich group}\label{descent}

In this section, we give a $3$-descent argument following \cite{satge}. Given any cube-free positive integer $\lambda=3^{e(\lambda)}D$ with $(D,3)=1$, we write 
$$E_\lambda:  y^2z=x^3+kz^3,$$
where 
$$
k =  \left\{ \begin{array}{ll}
-2^43^3\lambda^2&\mbox{if $e(\lambda)=0$ or $1$,} \\
-2^43D^2   & \mbox{if $e(\lambda)=2$.}
       \end{array} \right.
$$
This is birationally equivalent to $C_\lambda: x^3+y^3=\lambda$, and over $\Q$ it is $3$-isogenous to 
$$E_\lambda': y^2z=x^3+2^4\lambda^2z^3.$$
We have an isogeny $\phi: E_\lambda \to E_\lambda'$ explicitly given by
$$
\phi(x,y,z)= \left\{ \begin{array}{ll}
(3^{-2}(x^4-2^2kxz^3), 3^{-3}y(x^3+2^3kz^3),x^3z)&\mbox{if $e(\lambda)=0$ or $1$} \\
(x^4+4kxz^3, y(x^3-2^3kz^3), x^3z) & \mbox{if $e(\lambda)=2$}
       \end{array} \right.
$$
and the kernel of $\phi$ is $\{(0,1,0), (0, \pm \sqrt{k}, 1)\}$.
We write $\widehat{\phi}: E_\lambda'\to E_\lambda$ for the dual isogeny. Then, writing $k'=2^4\lambda^2$, we have
$$
\widehat{\phi}(x,y,z)=  \left\{ \begin{array}{ll}
(x^4+2^2k'xz^3, y(x^3-2^3k'z^3), x^3z) &\mbox{if $e(\lambda)=0$ or $1$} \\
(3^{-2}(x^4-2^2k'xz^3), 3^{-3}y(x^3+2^3k'z^3),x^3z)& \mbox{if $e(\lambda)=2$},
       \end{array} \right.
$$
the kernel of $\widehat{\phi}$ is $\{(0,1,0), (0, \pm 2^2\lambda, 1)\}$ and $\widehat{\phi}\circ \phi =[3]$ gives the multiplication-by-$3$-map. Recall that given any isogeny $\varphi: A\to A'$ between elliptic curves defined over a number field $F$, we have the exact sequence:
\begin{equation}\label{sel}
0\to A'(F)/\varphi  A(F)\to \Sel_\varphi(A/F)\to \Sha(A/F)[\varphi]\to 0
\end{equation}
where $\Sel_\varphi(A/F)$ denotes the $\varphi$-Selmer group of $A/F$ and $\Sha(A/F)[\varphi]$ denotes the $\varphi$-torsion part of the Tate--Shafarevich group of $A/F$. 

Before we state the next result, we recall the following useful hypothesis from \cite[p. 312]{satge}.

\begin{hyp}[H] For each prime divisor $p_i$ of $\lambda$ such that $p_i\equiv 1\bmod 3$, $p_j$ is a cube modulo $p_i$ for every $i\neq j$.
\end{hyp}

\begin{thm}\label{3descent} Let $C_\lambda: x^3+y^3=\lambda$, where $\lambda > 1$ is a cube-free integer. Let $s(\lambda)$ (resp. $r(\lambda)$) be the number of distinct primes dividing $\lambda$ which are congruent to $1$ (resp. $2$) modulo $3$. Let $d$ denote the number of distinct prime divisors $p$ of $\lambda$ such that $p\equiv 1\bmod 3$ and $3$ is not a cube modulo $p$. Then the $3$-Selmer group $\Sel_3(C_\lambda/\Q)$ of $C_\lambda$ over $\Q$ satisfies:
\begin{enumerate}
\item Suppose Hypothesis (H) holds, $s(\lambda)=d$, and one of the following conditions holds.
\begin{enumerate}[(i)]
\item $e(\lambda)=1$.
\item $e(\lambda)=2$ and there is an inert prime divisor $p\mid \lambda$ such that $p\neq 3$ and $p\not\equiv  8 \bmod 9$.

\item $e(\lambda)=0$, $s(\lambda)=0$ and there is a prime divisor $p\mid \lambda$ such that $p\not\equiv 8 \bmod 9$.
\end{enumerate}
Then,
 $$\dim_{\mb{F}_3}\left(\Sha(C_\lambda/\Q)[3]\right)=r(\lambda)-t(\lambda)-1-\rank(C_\lambda),$$
where $t(\lambda)=1$ if $\lambda\equiv \pm 1 \bmod 9$ (note that we have $r(\lambda)\geq 2$ automatically in this case), $t(\lambda)=-1$ if $\ord_3(\lambda)=1$ and $t(\lambda)=0$ otherwise. 
\item If $e(\lambda)=2$ $s(\lambda)=d$ and $p \equiv 8\bmod 9$ for every prime divisor $p\mid (\lambda/9)$ such that $p\equiv 2\bmod 3$, then 
$$
\dim_{\mb{F}_3}\left(\Sha(C_\lambda/\Q))[3]\right) =  r(\lambda)+1-\rank(C_\lambda).
$$
\item If $e(\lambda)=0$ and and $p\equiv 8\bmod 9$ for every prime divisor $p\mid \lambda$, then
$$\dim_{\mb{F}_3}\left(\Sha(C_\lambda/\Q)[3]\right) = r(\lambda)-\rank(C_\lambda).$$

\end{enumerate}

\end{thm}

\begin{proof} We consider the maps
\begin{equation*}\Sel_\phi(E_\lambda/\Q) \xto{\widehat{\phi}}\Sel_3(E_\lambda/\Q)\xto{\phi}\Sel_{\widehat{\phi}}(E_\lambda'/\Q).
\end{equation*}
We apply the snake lemma to \eqref{sel} with $\varphi=\phi$ and $[3]$. Note that the kernel of the map $E_\lambda'(\Q)/\phi E_\lambda(\Q)\xto{\widehat{\phi}} E_\lambda(\Q)/3 E_\lambda(\Q)$ is given by $E_\lambda'[\widehat{\phi}](\Q)/\phi E_\lambda[3](\Q)$ while the map $\Sha(E_\lambda/\Q)[\phi]\to \Sha(E_\lambda/\Q)[3]$ is injective. Thus, we see that $\ker\widehat{\phi}|_{\Sel_\phi(E_\lambda/\Q)}\simeq E_\lambda'[\widehat{\phi}](\Q)/\phi E_\lambda[3](\Q)$. Similarly, applying the snake lemma to \eqref{sel} with $\varphi=[3]$ and $\widehat{\phi}$ gives the cokernel $\Sel_{\widehat{\phi}}(E_\lambda'/\Q)/\phi\Sel_3(E_\lambda/\Q)\simeq \Sha(E_\lambda'/\Q)[\widehat{\phi}]/\phi \Sha(E_\lambda/\Q)[3]$. It follows (see also \cite[Lemma 5.1]{HYS}) that
\begin{equation}\label{sel3}\#\left(\Sel_3(E_\lambda/\Q)\right)=\frac{\#\left(\Sel_\phi(E_\lambda/\Q)\right)\cdot \#(\Sel_{\widehat{\phi}}(E_\lambda'/\Q))}{\#\left(E_\lambda'[\widehat{\phi}](\Q)/\phi E_\lambda[3](\Q)\right)\cdot \#\left(\Sha(E_\lambda'/\Q)[\widehat{\phi}]/\phi \Sha(E_\lambda/\Q)[3]\right)}.
\end{equation}
We know that $E_\lambda(\Q)_{\mathrm{tor}}=0$, $E_\lambda'(\Q)_{\mathrm{tor}}=E_\lambda'[\widehat{\phi}]=\{(0,1,0),(0,\pm 2^2 \lambda,1)\}\simeq \Z/3\Z$. Thus $E_\lambda'[\widehat{\phi}](\Q)/\phi E_\lambda[3](\Q)\simeq \Z/3\Z$.

Now,  the result of Satg\'{e} \cite[Proposition 2.8]{satge} based on the work of Cassels \cite{cassels} gives us
$$
\dim_{\mb{F}_3}(\Sel_\phi(E_\lambda/\Q))-\dim_{\mb{F}_3}(\Sel_{\widehat{\phi}}(E_\lambda'/\Q))=r(\lambda)-t(\lambda).
$$
It follows from \eqref{sel3} that
$$\dim_{\mb{F}_3}(\#(\Sel_3(E_\lambda/\Q)))\geq r(\lambda)-t(\lambda)-1.$$
Suppose now that one of the conditions (i)--(iii) holds in Theorem \ref{3descent} (1). Then by \cite[Th\'{e}or\`{e}me 2.9 (3)]{satge}, we have $\Sel_\phi(E_\lambda/\Q)\simeq (\Z/3\Z)^{r(\lambda)-t(\lambda)}$ and $\Sel_{\widehat{\phi}}(E_\lambda'/\Q)$ is trivial. Thus we must have $\Sha(E_\lambda'/\Q)[\widehat{\phi}]/\phi \Sha(E_\lambda/\Q)[3]$ trivial and $\ord_3(\#(\Sel_3(E_\lambda/\Q)))= r(\lambda)-t(\lambda)-1.$

For (2) and (3), we note that in \cite[Th\'{e}or\`{e}me 2.6]{satge}, the $(s(\lambda)+r(\lambda)+1)$-tuple should be a $(s(\lambda)+r(\lambda))$-tuple if $e(\lambda)=0$ (see \cite[Proposition 2.3 (3)]{satge}), so where it says $\dim_{\F_3}(\Sel_\phi(E_\lambda/\Q))\geq r(\lambda)$ at the end of the statement of the theorem (in the notation of \cite[Th\'{e}or\`{e}me 2.6]{satge}, $d(S)>c$), it should read $\dim_{\F_3}(\Sel_\phi(E_\lambda/\Q))\geq r(\lambda)-1$ if $e(\lambda)=0$. Using this, we obtain
$$r(\lambda)\leq \dim_{\mb{F}_3}\left(\Sel_3(C_\lambda/\Q)\right)\leq r(\lambda)+1$$ 
in case (2), and
$$r(\lambda)-1 \leq \dim_{\mb{F}_3}\left(\Sel_3(C_\lambda/\Q)\right)\leq r(\lambda)$$ 
in case (3). By the $3$-parity conjecture, which was proven in \cite{dd}, we know that $\dim_{\mb{F}_3}\left(\Sel_3(C_\lambda/\Q)\right)$ is even if and only if the global root number satisfies $\epsilon(C_\lambda/\Q)=+1$. On the other hand, by Proposition \ref{proproot} and \cite[Propostion A.2]{yukatam} we have 
$$\epsilon(C_\lambda/\Q)=(-1)^{r(\lambda)+1}$$
in case (2) and 
$$\epsilon(C_\lambda/\Q)=(-1)^{r(\lambda)}$$
in case (3). This concludes the proof of the theorem.
\end{proof}

\begin{remark}
\begin{enumerate}
\item The condition in Theorem \ref{3descent} (1) holds, for example, when $s(\lambda)=0$ and $e(\lambda)=1$, or in the following examples which satisfy $\rank(C_\lambda/\Q)=0$.
\begin{tiny}

\begin{equation}\nonumber
\begin{array}{lcccc}
\lambda & (t(\lambda), r(\lambda),s(\lambda))&\Sha(C_\lambda/\Q)[6]\\
\hline
3^2\cdot \cdot 29 \cdot 181 & (0, 1, 1)  &    \text{ trivial}\\
3^2\cdot 7 \cdot 29 \cdot 181 & (0, 1, 2)  &    \text{ trivial}\\
3^2\cdot 7^2 \cdot 29 \cdot 181 & (0, 1, 2)  &    (\Z/2\Z)^2\\
3^2\cdot 7 \cdot 113 \cdot 181 & (0, 1, 2)   &    \text{ trivial}\\
3^2\cdot 7 \cdot 29 \cdot 71\cdot 113 & (0, 3, 1)  &  (\Z/3\Z)^2\\
3^2\cdot 7^2 \cdot 29 \cdot 71\cdot 113 & (0, 3, 1)  &  (\Z/3\Z)^2\\
3^2\cdot  29 \cdot 71\cdot 113 \cdot 181 & (0, 3, 1)  &  (\Z/3\Z)^2\\
3^2\cdot 7^2 \cdot 29 \cdot 71\cdot 113 & (0, 3, 1)  &  (\Z/3\Z)^2\\
3^2\cdot 7 \cdot 29^2 \cdot 71\cdot 113 \cdot 181 & (0, 3, 2)  &  (\Z/3\Z)^2\\
3^2\cdot 7 \cdot 29^2 \cdot 71\cdot 113^2 \cdot 181 & (0, 3, 2)   &  (\Z/3\Z)^2\\
\hline
\end{array}
\end{equation}
\end{tiny}

\item Below we list some examples where the condition in Theorem \ref{3descent} (2) holds and $\rank(C_\lambda)=0$.
\begin{tiny}

\begin{equation}\nonumber
\begin{array}{lcccc}
\lambda & (r(\lambda),s(\lambda))& \Sha(C_\lambda/\Q)[6]\\
\hline
3^2\cdot 17^2  & (1, 0)  & (\Z/3\Z)^2\\
3^2\cdot 53^2  & (1, 0)   & (\Z/3\Z)^2\\
3^2\cdot 71^2  & (1, 0)  & (\Z/3\Z)^2\\
3^2\cdot 7^2\cdot 71\cdot 181 & (1,2)  &  (\Z/3\Z)^2\\
3^2\cdot 17^2 \cdot 53 \cdot 71 & (3, 0)   & (\Z/3\Z)^4\\
3^2\cdot 17 \cdot 53 \cdot 71^2 & (3, 0)    & (\Z/3\Z)^4\\
3^2\cdot 53 \cdot 71^2 \cdot 89 & (3, 0)    & (\Z/3\Z)^4\\
3^2\cdot 71 \cdot 89 \cdot 107 & (3, 0)   & (\Z/3\Z)^4\\
3^2\cdot 17^2 \cdot 53 \cdot 71\cdot 89\cdot 107 & (5, 0)  & (\Z/2\Z)^2\oplus (\Z/3\Z)^6\\
\hline
\end{array}
\end{equation}
\end{tiny}

\end{enumerate}

\end{remark}

\section{Numerical Examples}\label{numerical}

The root number of $C_\lambda/\Q$ was computed already in, for example, \cite[Proposition A.2]{yukatam} when $(3,\lambda)=1$. The following is obtained easily by combining the results of Rohrlich \cite{roh} and V\'{a}rilly-Alvarado \cite{VA}, and is included for the convenience of the reader.

\begin{prop}\label{proproot}  Let $\lambda>0$ be any cube-free integer, and let $e(\lambda)\in \{0, 1,2\}$ be such that $\lambda=3^{e(\lambda)}D$ and $(3,D)=1$. Then the global root number $\epsilon(C_\lambda/\Q)$ of $C_\lambda/\Q$ satisfies 
$$\epsilon(C_\lambda/\Q)=(-1)^{r(\lambda)-t(\lambda)-1},$$
where $r(\lambda)$ is the number of prime divisors of $\lambda$ which congruent to $2$ modulo $3$, $t(\lambda)=1$ if $\lambda\equiv \pm 1\bmod 9$, $t(\lambda)=1$ if $e(\lambda)=1$ and $t(\lambda)=0$ otherwise.
\end{prop}
\begin{proof} Given a prime number $q$, let $\epsilon_q(C_\lambda/\Q)$ denote the local root number of $C_\lambda/\Q$ at $q$. Noting $-2^4D^2\equiv 2\bmod 3$ and applying \cite[Lemma 4.1]{VA}, we have $\epsilon_3(C_\lambda/\Q)=-1$ if $\lambda\equiv \pm 1\bmod 9$ or $e(\lambda)=1$ and $\epsilon_3(C_\lambda/\Q)=+1$ if $e(\lambda)=2$. Furthermore, we have $\epsilon_\infty(C_\lambda/\Q)=-1$ and for each prime divisor $p$ of $D$, we have $\epsilon_p(C_\lambda/\Q)=\left(\frac{-3}{p}\right)$ (see \cite[Lemma 4.1]{VA} for $p=2$ and \cite[Proposition 2]{roh} for odd $p$). The result now follows from the fact that the global root number is given as the product of local root numbers.
\end{proof}

The numerical examples below are listed in the increasing order of $r(\lambda)-e(\lambda)+1$ which is a lower bound for $\Sha(C_{\lambda})[3]$, then $n(\lambda)$, $s(\lambda)$ and $r(\lambda)$.\\
\vspace{-10pt}

\begin{tiny}

\begin{equation}\nonumber
\begin{array}{lcccc}
\lambda & (e(\lambda), r(\lambda),s(\lambda)) &  r(\lambda)-e(\lambda)+1 & L^{(\mathrm{\mathrm{alg}})}(C_\lambda,1)& \Sha(C_\lambda)\\
\hline
3 & (1,0,0) & 0 & 1  & \text{ trivial}\\ 
3^2 \cdot 2 & (2,1,0) & 0 & 1 & \text{ trivial}\\
3^2 \cdot 2^2 & (2,1,0) & 0 & 1 & \text{ trivial}\\
3^2 \cdot 5 & (2,1,0) & 0 & 1  & \text{ trivial}\\ 
3^2 \cdot 5^2 & (2,1,0) & 0 & 1 & \text{ trivial}\\ 
3 \cdot 7& (1,0,1) &  0& 3  & \text{ trivial}\\ 
3 \cdot 7^2& (1,0,1) &  0& 3  & \text{ trivial}\\ 
3 \cdot 13& (1,0,1) &  0& 3  & \text{ trivial}\\ 
3\cdot 31 & (1,0,1) & 0 & 3   &  \text{ trivial}\\ 
3^2\cdot 2\cdot 7^2 & (2,1,1) & 0 & 3 &  \text{ trivial}\\
3^2\cdot 2^2\cdot 7^2 & (2, 1, 1) & 0  & 3 &\text{ trivial}\\  
3^2 \cdot 5^2 \cdot 7^2  & (2,1,1) & 0 & 2^2\cdot 3  & (\Z/2\Z)^2\\ 
3^2\cdot 5\cdot 13^2 & (2,1,1) & 0 &   2^2\cdot 3 &  (\Z/2\Z)^2\\ 
3^2\cdot 7\cdot 11 & (2,1,1) & 0 & 3 &  \text{ trivial}\\ 
3^2\cdot 2\cdot 7^2\cdot 13 & (2,1,2) & 0 & 3^2 & \text{ trivial}\\ 
3^2\cdot 2\cdot 7^2\cdot 13^2 & (2,1,2) & 0 & 3^2 & \text{ trivial}\\ 
3^2\cdot 2^2\cdot  7^2\cdot 13^2  & (2,1, 2) & 0  & 2^2\cdot 3^2 &  (\Z/2\Z)^2\\  
3\cdot 7\cdot 13 \cdot 19^2 & (1,0,3) & 0 & 3^3 &  \text{ trivial}\\
3^2\cdot 2\cdot 7\cdot 13\cdot 19^2 & (2,1,3) & 0 & 3^3 & \text{ trivial}\\
3^2\cdot 2\cdot 7^2\cdot 13^2 \cdot 19 & (2,1,3) & 0 & 2^2\cdot 3^3 & (\Z/2\Z)^2\\   
3^2\cdot 2^2\cdot 7^2\cdot 13^2 \cdot 19 & (2,1,3) & 0 & 3^3 & (\Z/2\Z)^2\\   
3^2\cdot 7\cdot 11\cdot 13 \cdot 19 & (2,1,3) & 0 & 3^3 &  \text{ trivial}\\
3^2\cdot 7\cdot 13\cdot 19 \cdot 29 & (2,1,3) & 0 & 3^3 & \text{ trivial}\\
3\cdot 7\cdot 13^2\cdot 19\cdot 31 & (1,0,4) & 0 & 3^4 & \text{ trivial}\\ 
3\cdot 2\cdot5^2 & (1, 2, 0) &2  & 3^2& (\Z/3\Z)^2\\  
3\cdot 2^2\cdot5^2 & (1, 2, 0) &2  & 3^2& (\Z/3\Z)^2\\  
3\cdot 2 \cdot 11 & (1, 2, 0) &2  & 3^2& (\Z/3\Z)^2\\  
3\cdot 2 \cdot 11^2 & (1, 2, 0) &2  & 3^2& (\Z/3\Z)^2\\  
3\cdot 2^2 \cdot 11^2 & (1, 2, 0) &2  & 3^2& (\Z/3\Z)^2\\  
3\cdot 5\cdot 11 & (1,2,0) & 2 & 3^2 &  (\Z/3\Z)^2\\  
3\cdot 17\cdot 29 & (1,2,0) &  2 &   3^2 &  (\Z/3\Z)^2\\ 
3^2 \cdot 2 \cdot 5 \cdot 11 & (2,3,0) & 2 & 3^2&(\Z/3\Z)^2\\  
3^2 \cdot 2^2 \cdot 5^2 \cdot 11^2 & (2,3,0) & 2 & 3^2&(\Z/3\Z)^2\\  
3^2\cdot 5^2 \cdot 11\cdot 15 & (2,3,0)&2 & 3^2 & (\Z/3\Z)^2\\ 
3\cdot 5\cdot 7\cdot 11\cdot 13 & (1,2,2) & 2 & 3^4  &  (\Z/3\Z)^2\\ 
3^2\cdot 2^2\cdot 7^2\cdot 11^2\cdot 13^2 & (2,3,2) & 2  &  3^4& (\Z/3\Z)^2\\ 
3\cdot 2\cdot 5\cdot 11\cdot 17 & (1,4,0) & 4 & 3^4 &  (\Z/3\Z)^4\\ 
3\cdot 2\cdot 5\cdot 11^2 \cdot 17 & (1,4,0) & 4 & 3^4 &  (\Z/3\Z)^4\\ 
3\cdot 2\cdot 5\cdot 11 \cdot 17^2 & (1,4,0) & 4 & 3^4 &  (\Z/3\Z)^4\\ 
3\cdot 5\cdot 11\cdot 17\cdot 23 & (1,4,0) & 4& 3^4  &  (\Z/3\Z)^4\\ 
3^2\cdot 2\cdot 5\cdot 11\cdot 17\cdot 23 & (2, 5,0) & 4  & 3^4 & (\Z/3\Z)^4 \\
3^2\cdot 2\cdot 5\cdot 11^2 \cdot 17\cdot 23 & (2, 5,0) & 4  & 3^4 & (\Z/3\Z)^4 \\
3^2\cdot 2\cdot 5\cdot 11 \cdot 29\cdot 59 & (2, 5,0) & 4  & 3^4 & (\Z/3\Z)^4 \\
3\cdot 2\cdot 5\cdot 7^2\cdot 11\cdot 17 & (1,4,1) & 4 & 3^5 & (\Z/3\Z)^4 \\
3\cdot 2\cdot 5\cdot 11\cdot 13\cdot 17 & (1,4,1) & 4  & 3^5 &  (\Z/3\Z)^4\\ 
3\cdot 2^2\cdot 5\cdot 13\cdot 17\cdot 23 & (1,4,1) & 4  & 3^5 &  (\Z/3\Z)^4\\ 
3\cdot 2^2\cdot 5\cdot 11\cdot 17\cdot 31 & (1,4,1) & 4  & 3^5 &  (\Z/3\Z)^4\\ 
3^2\cdot 2^2 \cdot 5 \cdot 7 \cdot 11 \cdot 17 \cdot 23 & (2, 5,1) & 4  & 3^5 & (\Z/3\Z)^4 \\
3\cdot 5^2 \cdot 7\cdot 11\cdot 13\cdot 17\cdot 23 & (1, 4, 2) & 4 & 3^6 & (\Z/3\Z)^4 \\
3\cdot 2\cdot 5\cdot 11^2\cdot 17\cdot 23\cdot 29 & (1, 6, 0) & 6 & 3^6 &  (\Z/3\Z)^6 \\
3\cdot 2\cdot 5\cdot 11\cdot 17^2\cdot 23\cdot 29 & (1, 6, 0) & 6 & 3^6 &  (\Z/3\Z)^6 \\
3\cdot 2\cdot 5\cdot 11\cdot 17^2\cdot 23\cdot 41 & (1, 6, 0) & 6 & 3^6 &  (\Z/3\Z)^6 \\

\hline

\end{array}
\end{equation}
\end{tiny}

 \section*{Acknowledgement}
I would like to thank John Coates, Yongxiong Li and Don Zagier for their interest, comments and encouragement. I would also like to thank the referee for carefully reading the manuscript and for the helpful comments. I am grateful to the Max Planck Institute for Mathematics in Bonn for its support and hospitality. This project was partially supported by the European Union’s Horizon 2020 research and innovation programme under the Marie Sk\l{}odowska-Curie grant agreement No.~101026826.

\bibliographystyle{amsalpha}

\begin{thebibliography}{BCP97}




\bibitem{BSD} B. Birch, P. Swinnerton-Dyer,  {\em Notes on elliptic curves. {II}}, J. Reine Angew. Math. 218 (1965), 79--108.

\bibitem{magma}W. Bosma, J. Cannon, and C. Playoust, {\em The Magma algebra system. I. The user language}, J. Symbolic Comput., 24 (1997), 235--265.

\bibitem{BCDT} C. Breuil, B. Conrad, F. Diamond, and R. Taylor, \emph{On the modularity of
elliptic curves over $\Q$: wild $3$-adic exercises}, J. Amer. Math. Soc., 14(4) (2001), 843--939.

\bibitem{cassels62} J.W.S. Cassels, {\em Arithmetic on curves of genus 1. IV. Proof of the Hauptvermutung},   J. Reine Angew. Math. 211 (1962), 95--112.



\bibitem{cassels} J.W.S. Cassels, {\em Arithmetic on curves of genus 1, VI. The Tate-Shafarevich group can be arbitrarily large}, J. Reine Angew.
Math. 214/215 (1964), 65--70.


\bibitem{coates1} J. Coates,  {\em Lectures on the Birch--Swinnerton-Dyer conjecture}, ICCM Not. 1 (2013), no. 2, pp.~29--46.

\bibitem{CW} J. Coates, A. Wiles  \textit{On the conjecture of Birch and Swinnerton-Dyer}, Invent. Math. 39 (3) (1977), pp.~223--251.

\bibitem{dd}T. Dokchitser, V. Dokchitser, \textit{On the Birch–Swinnerton-Dyer quotients modulo squares}, Ann. of Math. (2) 172 (2010),
no. 1, pp.~567--596.

\bibitem{gol-sch} C. Goldstein and N. Schappacher, {\em S\'eries d'{E}isenstein et fonctions {$L$} de courbes elliptiques \`a multiplication complexe}, J. Reine Angew. Math. 327 (1981) 184--218.

\bibitem{GZ} B. Gross, D. Zagier, \textit{Heegner points and derivatives of $L$-series}, Invent. Math. 84 (2) (1986), pp.~225–320.

\bibitem{HYS}Y.\ Hu, H.\ Yin, J.\ Shu, {\em   An explicit Gross-Zagier formula related to the Sylvester conjecture},   Trans.\ Amer.\ Math.\ Soc.\ 372 (2019), no.\ 10, 6905--6925.


\bibitem{yukatam}Y. Kezuka, {\em   Tamagawa number divisibility of central $L$-values of twists of the Fermat elliptic curve}, to appear. J. Th\'{e}or. Nombres Bordeaux.


\bibitem{kol} V. A. Kolyvagin, \textit{Euler systems},  The Grothendieck Festschrift (Vol. II), P. Cartier et al., eds., 435--483. Progress in Math. 87. Birkh\"{a}user, Boston 1990

 \bibitem{pari}The PARI~Group, PARI/GP version {\tt 2.11.1}, Univ. Bordeaux, 2018. \href{http://pari.math.u-bordeaux.fr}{http://pari.math.u-bordeaux.fr}




\bibitem {RZ}F.~Rodriguez Villegas, D.~Zagier \textit{Which primes are sums of two cubes?} CMS Conf. Proc., 15, Amer. Math. Soc., Providence, RI (1995).


\bibitem{roh}D. Rohrlich, {\em  Variation of the root number in families
of elliptic curves},    Compos. Math., 87, no.2 (1993), 119--151.

\bibitem{rubinsha}K. Rubin, ``Tate–Shafarevich groups and $L$-functions of elliptic curves with complex multiplication'', Invent. Math., 89 (3) (1987), 527--559.

\bibitem{rubin} K. Rubin, {\em The ``main conjectures'' of Iwasawa theory for imaginary quadratic fields},  Invent. Math. 103 (1991), no. 1, 25--68.


\bibitem{satge}P. Satg\'{e}, {\em    Groupes de Selmer et corps cubiques},    J. Number Theory 23 (1986), no. 3, 294--317.

\bibitem{selmer} E.S. Selmer, {\em The diophantine equation $ax^3 + by^3 + cz^3 = 0$}, Acta Math. 85 (1951), 203--362.

\bibitem{Shimura}G.~Shimura, {\em On elliptic curves with complex multiplication as factors of the Jacobians of modular function fields}, Nagoya Math. J.~43 (1971), 199--208.

\bibitem{sil}J. Silverman, {\em   The arithmetic of elliptic curves},   Graduate Texts in Mathematics, 106, 2nd edition. Springer-Verlag, 2000.
\bibitem{silverman}J. Silverman, {\em   Advanced topics in the arithmetic of elliptic curves},   Graduate Texts in Mathematics, 151. Springer-Verlag, New York, 1994, xiv+525, ISBN: 0-387-94328-5

\bibitem{stephens}N. Stephens, {\em   The diophantine equation $X^3+Y^3=DZ^3$ and the conjectures of Birch and Swinnerton-Dyer},    J. Reine Angew. Math. 231 (1968), 121--162.

\bibitem{sylvester}J.J. Sylvester, {\em On certain ternary cubic-form equations}, Amer. J. Math. 2 (1879), no. 4, 357--393. 

\bibitem{tate}J. Tate,  {\em Algorithm for determining the type of a singular fiber in an elliptic pencil}, Modular Functions of One Variable IV, Lecture Notes in Math., Springer 476 (1975) 33--52.

\bibitem{TW} R. Taylor and A. Wiles. \emph{Ring-theoretic properties of certain Hecke algebras}, Ann.
of Math. (2), 141(3) (1995), 553--572.

\bibitem{VA} A. V\'{a}rilly-Alvarado, {\em Density of rational points on isotrivial rational elliptic surfaces},  Algebra Number Theory, 5, no. 5 (2011), 659--690.

\bibitem{wil} A. Wiles, \emph{Modular elliptic curves and Fermat’s last theorem}, Ann. of Math. (2),
141(3) (1995), 443--551.





\end{thebibliography}

\newcommand{\etalchar}[1]{$^{#1}$}
\providecommand{\bysame}{\leavevmode\hbox to3em{\hrulefill}\thinspace}
\providecommand{\MR}{\relax\ifhmode\unskip\space\fi MR }

\providecommand{\MRhref}[2]{
  \href{http://www.ams.org/mathscinet-getitem?mr=#1}{#2}
}
\providecommand{\href}[2]{#2}

\end{document}